\documentclass[12pt]{amsart}

\usepackage{accents} 
\DeclareMathAccent{\wtilde}{\mathord}{largesymbols}{"65}

\usepackage{srcltx}  

\usepackage{geometry} 
\usepackage{amsmath}
\usepackage{amsfonts,amssymb,amsthm,amscd,latexsym,euscript}

\usepackage{color}
\usepackage{xypic}
\xyoption{all}
\def\umono{\ar@{_{(}->}[u]}
\def\uumono{\ar@{_{(}->}[uu]}

\def\lmono{\ar@{_{(}->}[l]}
\def\llmono{\ar@{_{(}->}[ll]}

\newcommand{\Z}{{\mathbb Z}}
\newcommand{\eS}{{\mathcal {S}}}

\newcommand{\F}{{\mathbb F}}

\newcommand{\op}{{\mathrm{op}}}

\newcommand{\nil}[1]{\operatorname{nil_{BD}}(#1)}
\newcommand{\nilbg}[1]{\operatorname{nil_{BG}}(#1)}
\newcommand{\epfl}[1]{\operatorname{epfl}(#1)}
\newcommand{\cocat}[1]{\operatorname{cocat}(#1)}

\newcommand{\CW}{\operatorname{CW}}
\renewcommand{\P}{\operatorname{P}}
\newcommand{\A}{\ifmmode{\mathcal{A}}\else${\mathcal{A}}$\fi}
\newcommand{\K}{\ifmmode{\mathcal{K}}\else${\mathcal{K}}$\fi}
\newcommand{\U}{\ifmmode{\mathcal{U}}\else${\mathcal{U}}$\fi}
\newcommand{\T}{\ifmmode{\mathcal{T}}\else${\mathcal{T}}$\fi}
\newcommand{\FF}{\ifmmode{\mathcal{F}}\else${\mathcal{F}}$\fi}
\newcommand{\LL}{\ifmmode{\mathcal{L}}\else${\mathcal{L}}$\fi}

\newtheorem{theorem}{Theorem}[section]
\newtheorem{proposition}[theorem]{Proposition}
\newtheorem{corollary}[theorem]{Corollary}
\newtheorem{lemma}[theorem]{Lemma}

\theoremstyle{definition}
\newtheorem{definition}[theorem]{Definition}
\newtheorem{example}[theorem]{Example}
\theoremstyle{remark}
\newtheorem{remark}[theorem]{Remark}

\def\N{{\mathbb{N}}}

\DeclareMathOperator{\Nil}{{\text{Nil}}}

\title[A Torus Theorem for loop spaces]{A Torus Theorem for homotopy \\ nilpotent loop spaces}

\author{Cristina Costoya, J\'er\^{o}me Scherer, and Antonio Viruel}

\thanks{The authors are supported by Xunta de Galicia grant EM2013/016. The first author is supported by Ministerio de Econom{\'\i}a y Competitividad (Spain), grant MTM2016-79661-P (AEI/FEDER, UE, support included).
The second author is supported by Ministerio de Econom{\'\i}a y Competitividad (Spain), grant MTM2013-42293-P. The third author is supported by Ministerio de Econom{\'\i}a y Competitividad (Spain), grants MTM2013-41768-P and MTM2016-78647-P (AEI/FEDER, UE, support included).}

\thanks{}

\subjclass[2010]{Primary 55P35; Secondary 55P65, 18C10, 55M30}
\keywords{nilpotent, homotopy nilpotent, cocategory, algebraic theory, Goodwillie calculus, excisive functor, $p$-compact group}

\begin{document}


\begin{abstract}
Nilpotency for discrete groups can be defined in terms of central extensions. In this paper, the analogous definition
for spaces is stated in terms of principal fibrations having infinite loop spaces as fibers,  yielding a new invariant
between the classical LS cocategory and the more recent notion of homotopy nilpotency introduced by
Biedermann and Dwyer. This allows us to characterize finite homotopy nilpotent loop spaces in the spirit of Hubbuck's
Torus Theorem, and  obtain corresponding results for $p$-compact groups and $p$-Noetherian groups.
\end{abstract}


\maketitle


\section*{Introduction}
\label{sec:intro}



Hubbuck's Torus Theorem, \cite[Theorem 1.1]{Hubbuck}, characterizes,  up to homotopy, classical homotopy commutative 
finite $H$-spaces as tori, i.e.\ finite products of circles. Mod $p$ versions were established after by Lin \cite[Theorem 1]{MR0808914} 
and Aguad\'e--Smith \cite[Corollary]{AguadeSmith}, and many extensions of this theorem have been obtained by relaxing the finiteness 
conditions. For example, Castellana--Crespo--Scherer \cite[Corollary~7.4]{CCS_1}  proved that connected homotopy commutative 
$H$-spaces with finitely generated cohomology as algebra over the Steenrod algebra,  at the prime~$2$, are homotopically equivalent 
to products of a torus and a Potsnikov piece. In this paper we relax instead the commutativity assumption by considering 
the Biedermann--Dwyer notion of homotopy nilpotent loop spaces and show that tori are the only homotopy nilpotent finite loop spaces.

Biedermann and Dwyer, \cite[Definition 5.4]{MR2580428},  used the stages of the Goodwillie tower
of the identity to provide the first definition of homotopy nilpotency which interpolates
between infinite loop spaces (homotopy  commutative groups or, equivalently, homotopy $1$-nilpotent groups) and loop spaces. Even though this might come as a surprise to those which are more
accustomed to relate the Goodwillie tower of the identity to $v_n$-periodicity, as
discovered by Arone and Mahowald, \cite{MR1669268}, there is a quite straightforward
relationship between the essence of Goodwillie calculus, namely higher excision, and
nilpotency.

In order to explain this let us first go back to group theoretical nilpotency.
The nilpotency of a discrete group is understood either as the minimal number of iterated commutators
that must always vanish, or as the minimal number of central extensions needed to construct the group.
The commutators approach  has been successfully interpreted in homotopy theory  leading to the classical Berstein--Ganea nilpotency \cite{MR0126277}. Work of Hopkins, \cite{MR1017164},
and Rao, \cite{MR1441487}, gives a complete understanding of the classical Berstein--Ganea nilpotency
for compact Lie groups: those with finite nilpotency index are precisely the torsion free ones. This is the case
of the $3$-dimensional sphere, an example which has been studied by Porter in the early 1960's, \cite{MR0169244}.

The central extensions approach will be this article's viewpoint. A homotopy
commutative
(or $1$-nilpotent) group is an infinite loop
space as it should be commutative not only up to homotopy, but up to all higher homotopies.  
Therefore a natural
analogue notion of nilpotency for loop spaces  (i.e. group like spaces) is the following formal translation.   We replace central extensions by principal fibrations whose fiber is an
infinite loop space and,
the invariant which is introduced in this way is the minimal number
of extensions by such principal fibrations needed to construct a given loop space.
We call this invariant the \emph{extension by principal fibrations length}, or epfl for short.  Let us mention that considering a central extension of groups as a principal fibration by taking classifying spaces,  is a classical procedure (see \cite[Lemma~IV.1.12]{MR2035696}).  It appears notably in algebraic K-theory, in Quillen's use of the plus construction,
a point of view adopted for example in \cite{MR0649409}.

The relationship to Goodwillie calculus is provided by the structure of the Goodwillie tower.
Let $F$ be a functor from spaces to spaces and $F\rightarrow P_n F$ denote the
$n$-excisive approximation of $F$, \cite[Section~1]{MR2026544}.
This means that $P_1 F$ is basically a homological functor turning homotopy push-outs into
homotopy pull-backs and more generally $P_n F$  satisfies higher excision, \cite[Definition~2.1]{MR1162445}.
The homotopy fiber $D_n F$ of the natural transformation $P_n F \rightarrow P_{n-1} F$ is
a homogeneous $n$-excisive functor, \cite[Proposition~1.17]{MR2026544}, and Goodwillie showed that it is
classified by a spectrum together with an action of the symmetric group $\Sigma_n$, \cite[Section 5]{MR2026544}.
Concretely $D_n F (X)$ is the infinite loop space corresponding to the $\Sigma_n$-equivariant smash product
of this spectrum with $X^{\wedge n}$. Even better, this fibration of functors is actually principal,
\cite[Lemma~2.2]{MR2026544}.

\medskip
In this paper, we prove the following.\\

\noindent {\bf Theorem~\ref{thm:inequalities}}
{\it
Let $Z$ be a pointed connected space. Then, we have the inequalities
\[
\nilbg {\Omega Z} \leq  \cocat Z  \leq \epfl  Z \leq \nil {\Omega Z}.
\]
}


Here the two nilpotency invariants refer to the classical Berstein--Ganea nilpotency,
$\operatorname {nil_{BG}}$ (see Section \ref{subsec:BG}), and the
Biedermann--Dwyer nilpotency, $\operatorname {nil_{BD}}$ (see Section \ref{subsec:BD}).

We mention finally two problems where our epfl invariant is key to the solution.
First we recover the vanishing of iterated Whitehead products in values of excisive functors.
If $F$ is an $n$-excisive functor from the category of
pointed spaces to pointed spaces and $K$ a finite space, then $\Omega F(K)$ is a homotopy nilpotent
group of class $n$, \cite[Corollary 9.3]{MR2580428}. Thus Theorem~\ref{thm:inequalities} immediately implies that all
$(n+1)$-fold iterated Whitehead  products vanish in $F(K)$. This has been proven originally in \cite{ChoSch}
relying on Goodwillie's generalized Blakers--Massey Theorem and the definition itself of an $n$-excisive functor.
In \cite{Eldred} Eldred obtains another proof by analyzing Goodwillie's
construction of the $n$-excisive approximation $P_nF$ of a functor $F$.
She shows in fact that Whitehead products already vanish in $T_n F$, a functor
directly related to Hopkins' symmetric cocategory, \cite[Definition p.\ 219]{Hopkins}.

The second problem, from which the title of this paper comes from, concerns our understanding of finite homotopy nilpotent loop spaces, 
with which we come back to the first lines of this introduction.
We recall briefly that a loop space $(X, BX)$ consist of a pair of pointed spaces $X$ and $BX$, together with
a homotopy equivalence $e:\Omega BX\simeq X$ defining a loop structure on $X$. The space $BX$ is called the classifying space
of $X$ and $(X, BX)$ is said to be finite if $H^*(X;\Z)$ is finitely
generated as a graded abelian group.  In this paper we offer the following version
of Hubbuck's Torus Theorem for finite loop spaces. Whereas the original statement,
\cite[Theorem 1.1]{Hubbuck}, focuses on homotopy commutative $H$-spaces, we
deal with arbitrary homotopy nilpotent groups.

\medskip
\noindent {\bf Theorem~\ref{thm:A}}
{\it
Let $(X, BX)$ be a connected finite loop space. If $\nil{X}$ is finite, then  $X$ has the homotopy type of a torus.
}

\medskip

We obtain analogous characterizations of Biedermann--Dwyer nilpotency for $p$-compact groups (Theorem \ref{thm:C}) and $p$-Noetherian groups (Corollary \ref{cor:thmB}). Two key ingredients in the proof are, on one hand, the existence of a finite tower of principal fibrations
and, on the other hand,  the effect of Neisendorfer's functor, \cite{MR1321002}, on the fibers of this tower making them contractible (see Section \ref{sec:infiniteloop}).

The result above shows that, in particular,  simple compact Lie groups are not nilpotent loop spaces in the sense of
Bierdermann--Dwyer, in contrast to the classical situation when considering Berstein--Ganea nilpotency where all torsion free
simple Lie groups are nilpotent \cite{Hopkins, MR1441487}. This provides yet more evidence that $\operatorname{nil_{BD}}$
is the correct notion for nilpotency of loop spaces.

\medskip
\noindent {\bf Acknowledgments.} The first and third author would like to acknowledge the hospitality of the EPFL,
where this paper has been completed. We would like to thank Georg Biedermann for sharing his unpublished results
with us, and the referees for their careful reading and constructive comments.
\section{Background and first results}
\label{sec:notation}
In this section we briefly introduce the main ingredients in Theorem~\ref{thm:inequalities}.

\subsection{Nilpotency in the sense of Biedermann--Dwyer}
\label{subsec:BD}
We need some basic notions on algebraic theories to understand the nilpotency in the sense of Biedermann--Dwyer.
Algebraic theories were introduced by Lawvere \cite{MR0158921} to describe algebraic structures, and successfully interpreted in homotopy theory by Badzioch
\cite{MR1923968}. We will need simplicial theories and follow
closely the viewpoint from \cite[Section 3]{MR2580428}.

\begin{definition}
\label{def:algth}An algebraic theory is a small category $T$  whose objects are indexed by the natural numbers
$\{T_{0}, T_{1},\ldots, T_{n},\ldots\}$ such that for $n\in \N$ the $n$-fold categorical coproduct of $T_{1}$ is
naturally isomorphic to $T_{n}$.
\end{definition}
In our situation, the algebraic theory $T$ will be required to be
 pointed and \emph{simplicial}, which means that $T$ is enriched over  the category $\eS_{\ast}$ of pointed simplicial sets.
We distinguish between \emph{strict} and \emph{homotopy} $T$-algebras:
simplicial functors $\underaccent{\wtilde}{X}\colon T^{\op}\to \eS_{\ast}$ taking
coproducts in $T^{\op}$ to products in $ \eS_{\ast}$ strictly or up to homotopy, respectively.

Biedermann and Dwyer define \emph{homotopy
nilpotent groups} as homotopy $\mathcal G_n$-algebras in the category
of pointed spaces, where $\mathcal G_n$ is a simplicial algebraic
theory constructed from the Goodwillie tower of the identity, \cite[Definition 5.4]{MR2580428}.
Concretely, ${\mathcal G_n}$ is the simplicial category which has for each natural number $k \geq 0$ exactly one
object given by $\mathcal G_n (k) = \underset{k}\prod \Omega (P_n (\operatorname{id}))^{\operatorname{inj}}$.
The functor $P_n(\operatorname{id})$ lives in the category of functors from finite pointed spaces to pointed spaces, it is
the $n$-excisive approximation of the identity functor, and
 $(P_n (\operatorname{id}))^{\operatorname{inj}}$ denotes the fibrant replacement in the injective model structure constructed 
 by Lurie in \cite[Proposition A.3.3.2]{MR2522659}, see also Joyal \cite{joyal}, Jardine \cite{jardine} for related work,
 and \cite[Section 4]{MR2580428} for the application to homotopy nilpotent groups.
The simplicial  set of morphisms $ {\mathcal G_n} (k, l)$ is the space of natural transformations
$\underset{k}{\prod}\Omega (P_n (\operatorname{id}))^{\operatorname{inj}} \rightarrow 
\underset{l}{\prod}\Omega (P_n (\operatorname{id}))^{\operatorname{inj}}$.

Hence, a pointed space X is a \emph{homotopy nilpotent group of class $\leq n$} if it is the value at $1$ of a simplicial functor
$\underaccent{\wtilde}{X}$ from $\mathcal G_n $  to pointed spaces, which commutes up to homotopy with products.
In fact, the homotopy nilpotent group is the whole functor $\underaccent{\wtilde}{X}$ and we isolate abusively the space $\underaccent{\wtilde}{X}(1)$, that comes in particular with a loop space structure since there is a structure map
$\underaccent{\wtilde}{X}(1) \times \underaccent{\wtilde}{X}(1) \simeq \underaccent{\wtilde}{X}(2) \rightarrow \underaccent{\wtilde}{X}(1)$ corresponding to the loop space product. We write $\nil X \leq n$ and, we say that $X$ is a \emph{homotopy nilpotent group} if it is of class $\leq n$ for some $n$.
The two extremes of this theory are well understood: loop spaces ($\nil X \leq \infty$) and infinite loop spaces ($\nil X \leq 1$, see \cite[Theorem 5.13]{MR2580428}) can be described as homotopy algebras
over the theories $\mathcal G_{\infty}$ and $\mathcal G_1$ respectively.

 \subsection{Nilpotency in the sense of Berstein--Ganea}
 \label{subsec:BG}
Applying the functor $\pi_0$ to the simplicial algebraic theory $\mathcal G_n$ gives us the ordinary theory of $n$-nilpotent groups,
$\operatorname{Nil}_n$, \cite[Theorem 8.1]{MR2580428}.  Now, product preserving functors to the homotopy category of pointed spaces,
$N \colon \operatorname{Nil_n^{op}} \rightarrow \operatorname {Ho(Spaces_\ast)},$
are, in other words, $\operatorname{Nil}_n$-algebras in  the homotopy category of pointed spaces. These are exactly the $n$-nilpotent
groups in the sense of Berstein--Ganea \cite[Proposition 4.2]{ChoSch}.  Recall that for a loop space $\Omega Z$,
the nilpotency of $\Omega Z$  in the sense of Berstein--Ganea is the least integer $n$ for which the $(n+1)$-st commutator map
$\varphi_{n+1}: {(\Omega Z)}^{n+1} \rightarrow \Omega Z$ is homotopically trivial \cite[Definition 1.7]{MR0126277}; in  this paper we write
$\nilbg {\Omega Z} \leq n$.  We recall here that if such an integer $n$ exists then all the $(n+1)$-fold iterated Whitehead products
vanish in $Z$, \cite[Theorem 4.6]{MR0126277}.

The following example shows that nilpotency in the sense of Berstein--Ganea does not capture
the subtlety of the loop space structure.
\begin{example}
\label{ex:splitting}
Let $f: \mathbb C P^\infty \rightarrow K(\mathbb Z, 6)$ represent the cube
of the fundamental class of the infinite complex projective space, and let $Z$ be the homotopy fiber of $ f$.
The Berstein--Ganea nilpotency of $\Omega Z$ is $1$ \cite[Corollary 5.6]{MV}  but its Biedermann--Dwyer nilpotency is strictly greater
than $1$ since $\Omega Z$ is not an infinite loop space ($Z$ is not an $H$-space  \cite[Lemma 5.2]{MV}). Indeed, even though 
$\Omega Z$ is homotopically equivalent to a product of infinite loop spaces, $\Omega Z \simeq S^1 \times K (\mathbb Z, 4)$,  
\cite[Lemma 5.4]{MV}, the homotopy does not preserve the loop structure.
\end{example}

The next two invariants are recalled in the following paragraph since they are both related to fibrations, and
in some sense, the second one is a refinement of the first.
\subsection{Inductive cocategory and epfl}
\label{sec:epfl}
Different notions of cocategory exist, see for example \cite[Definition 2]{Hovey}, \cite[Definition p.\ 219]{Hopkins}, and \cite[Definition 3.4]{MV}.  
In this paper we concentrate on the inductive cocategory  introduced by Ganea, \cite[Definition 2.1]{Ganea}, concretely, on its normalized version.
Thus, for a connected pointed space $Z$,  $\cocat  Z = 0$ if and only if $Z$ is contractible and, for any $n \geq 1$,  $\cocat Z \leq n$ if there 
exists a fibration $F \rightarrow Y \rightarrow B$ with $\cocat Y \leq n-1$ and $F$ dominates $Z$ (i.e. $Z$ is a retract of $F$);
it is clear from the definition that $\cocat F \leq \cocat Y +1$. It is also clear that $\cocat Z \leq 1$ if and only if $Z$ is dominated by a loop space.

The relation between the cocategory and the nilpotency in the sense of Berstein--Ganea is given by the inequality  $\nilbg {\Omega Z}  \leq \cocat Z$, \cite[Theorem 2.12]{Ganea}.

\medskip
In this paper, we introduce a variant of the previous definition.
\begin{definition}
\label{def:epfl}
We say that  a pointed connected space $Z$ is  \emph{an extension by principal fibrations of length $0$} if and only if $Z$ is contractible and,
for $n \geq 1$, a space $Z$ is an \emph{extension by principal fibrations of length $\leq n$},  if
there exists a tower  of principal fibrations
\begin{equation}\label{epflxy}
\xymatrix{
Z \simeq Z_n  \ar[r] & Z_{n-1} \ar[r] & \dots \ar[r] & Z_1 \ar[r] & Z_0 = \ast \\
F_n \ar[u] & F_{n-1} \ar[u] & & F_1\ar@{=}[u]
}
\end{equation}
where $Z_0$ is a point and all fibers $F_k = \textrm{Fib}(Z_k \rightarrow Z_{k-1})$ are infinite loop spaces.
In that case, we write $\epfl{Z} \leq n$.

\end{definition}

\begin{lemma}
\label{lemma:inequality1}
Let $Z$ be a pointed connected space. Then $\cocat Z \leq \epfl Z$.
\end{lemma}

\begin{proof}
We prove the lemma by induction on $\epfl Z$. If $\epfl Z = 1$, $Z$ is an infinite loop space, so $\cocat Z \leq 1$.
If $\epfl Z\leq n$ for some integer $n >1$, by definition there exists a tower of principal fibrations \eqref{epflxy}
 where the fibration
$Z \rightarrow Z_{n-1}$ is classified by  $\theta\colon Z_{n-1} \rightarrow B F_{n}$, that is, $Z$ is the homotopy fiber of $\theta$.
Hence, $\cocat Z\leq \cocat {Z_{n-1}} + 1$ and, by induction, we get that $\cocat Z \leq \cocat {\ast} + n =n$.
\end{proof}

We are now ready to prove the main theorem in this section. We exploit a characterization of homotopy nilpotent groups
through excisive functors. For if $F$ is  an  $n$-excisive functor (so $F$ sends strongly homotopy co-Cartesian $(n+1)$-cubes
to homotopy Cartesian ones, and also $F \simeq P_nF$ by \cite[Theorem~1.8]{MR2026544})
 then, for every finite space $K$, $\Omega F (K)$ is a homotopy nilpotent group \cite[Corollary 9.3]{MR2580428}.
 Even better,  Biedermann shows that every functor $\underaccent{\wtilde}{X}$ associated to a homotopy nilpotent group
 of class $\leq n$ is of the form $\Omega F$ where $F$ is $n$-excisive, \cite{Biedermann}.

\medskip
 We prove the following:

\begin{theorem}
\label{thm:inequalities}
Let $Z$ be a pointed connected space. Then, we have the inequalities
\[
\nilbg {\Omega Z} \leq  \cocat Z  \leq \epfl  Z \leq \nil {\Omega Z}.
\]
\end{theorem}

\begin{proof}
In view of the previous paragraphs, nothing needs to be done for the first two inequalities.
The first one is \cite[Theorem 2.12]{Ganea} and the second one is Lemma~\ref{lemma:inequality1}.
Suppose that $\Omega Z$  is a homotopy nilpotent group of class $\leq n$. Then, by \cite{Biedermann},
there exists an $n$-excisive functor $F$ and a space $K$
such that $\Omega Z$ and $\Omega F (K)$ are weakly equivalent as loop spaces.
Therefore we have an equivalence $Z \simeq F(K)$.
Now, as we mentioned in the Introduction, the Goodwillie tower for $F \simeq P_nF$  yields a tower
\[
F (K) \simeq P_n F (K) \rightarrow P_{n-1} F (K)  \rightarrow \cdots \rightarrow P_1 F (K) \rightarrow \ast
\]
whose fibers $D_k F (K)$ are infinite loop spaces, and $P_k F (K) \rightarrow P_{k-1}F (K)$ are principal
fibrations classified by  $P_{k-1}F (K) \rightarrow BD_k F (K)$, \cite[Lemma~2.2]{MR2026544}. This directly implies that $\epfl Z \leq n$.
\end{proof}

When $\Omega Z$ is not only a loop space, but an infinite loop space, all inequalities are indeed equalities. This
comes from the fact that homotopy $1$-nilpotent groups are infinite loop spaces \cite[Theorem 5.13]{MR2580428}.

\begin{corollary}
\label{prop:equalities}
Let $X$ be an infinite loop space. Then $\nilbg {X} = \nil {X} = 1$.
\end{corollary}

This approach leads us to an alternate and more direct proof of a result obtained in \cite[Theorem~2.1]{ChoSch}
(see also Eldred's point of view, \cite[Corollary~4.3]{Eldred}).

\begin{corollary}
\label{thm:whiteheadproductsvanish}
Let $F$ be any $n$-excisive functor from the category of
pointed spaces to pointed spaces. Then all $(n+1)$-fold iterated
Whitehead products vanish in $F(K)$ for every finite space $K$.
\end{corollary}

\begin{proof}
Since $F$ is an $n$-excisive functor, $\Omega F (K)$ is  a homotopy nilpotent group of class $\leq n$,  \cite[Theorem 9.2]{MR2580428}.
Thus $\nilbg  {\Omega F(K)}\leq n$ by Theorem~\ref{thm:inequalities} and all the $(n+1)$-fold iterated Whitehead products vanish
in $F(K)$ (see \S\ref{subsec:BG}).
\end{proof}

This application
highlights the simplicity of the arguments when using the relationship between the classical notions of nilpotency,
$\operatorname{nil_{BG}}$ and $\operatorname{cocat}$,  and the more recent ones,
$\operatorname{epfl}$ and $\operatorname{nil_{BD}}$.

\begin{remark}
\label{rem:better}
{\rm The upper bound we obtain in Corollary \ref{thm:whiteheadproductsvanish} is a crude estimate.  In  \cite[Theorem~4.2]{ADL} Arone, Dwyer, and Lesh show that any $(2n-1)$-excisive space-valued functor $F$,  for which $P_{n-1} F \simeq \ast$, takes its values in infinite loop spaces.  Therefore, since $\operatorname{nil_{BG}} F (K) \leq 1$,  all Whitehead products vanish in  $F(K)$,  whereas by Corollary \ref{thm:whiteheadproductsvanish}  we would only obtain that the $2n$-fold iterated Whitehead products vanish.
}
\end{remark}

We end this section by showing that all nilpotency notions coincide for discrete groups.
In order to do so, we use the relation of the simplicial algebraic theory $\mathcal G_n$ with the set-valued theory of ordinary
$n$-nilpotent groups~$\Nil_n$.
\begin{theorem}
\label{thm:solo-una-nilpotencia-para-grupos}
Let $G$ be a discrete nilpotent group. Then $$\nilbg G =\cocat {BG}
=\epfl {BG} =\nil G.$$
\end{theorem}

\begin{proof}
Let $G$ be a nilpotent group of class $n$, i.e. $n=\nilbg G$. There exists then a set valued $\Nil_n$-algebra given by the
functor
$\underaccent{\wtilde}{G}\colon \Nil_n\rightarrow Sets_*$ where $\underaccent{\wtilde}{G}(k)=G^k$.
Considering the isomorphism of categories $\Nil_n\cong\pi_0{\mathcal G}_n$
\cite[Theorem 8.1]{MR2580428} and viewing pointed sets as a full
subcategory of ${\mathcal S}_*$,  shows that $G$ is a homotopy nilpotent
group of class $\leq n$ since the functor $\underaccent{\wtilde}{G}\circ\pi_0 \colon {\mathcal G}_n\rightarrow
{\mathcal S}_*$ extends to a simplicial functor as all mapping spaces between discrete groups are discrete.
It is therefore an homotopy ${\mathcal G}_n$-algebra and
\[
n=\nilbg G \leq\cocat {BG} \leq\epfl {BG} \leq\nil G \leq n.
\]

\end{proof}

\begin{remark}
We would be remiss in not saying a word about examples of spaces for which these inequalities are sharp. All of our attempts have been unsuccessful to find an example of a space $X$  for which
$\epfl X$ is strictly less than $\nil {\Omega X}$.

\end{remark}



\section{Towers of principal fibrations and Neisendorfer's type functor}
\label{sec:infiniteloop}
The aim of this section is to show that the effect of certain homotopical loca\-lization functors on loop spaces with
finite epfl, see Definition \ref{def:epfl}, is predictable. We start with a short subsection where we fix the notation and
terminology about localization and cellularization. Most of this is taken from the first chapters in~\cite{Dror}. It is convenient
to work here in the category of simplicial sets (which we call spaces). From now on $p$ denotes a prime number.

\subsection{Localization and cellularization}
\label{subsec:loc}
Let $f\colon A \rightarrow B$ be a map. A space $X$ is $f$-\emph{local} if $\hbox{\rm map}(f, X)$ is a weak equivalence.
A map $g$ is an $f$-\emph{local equivalence} if $\hbox{\rm map}(g, X)$ is a weak equivalence for all $f$-local spaces $X$.
There exists a coaugmented homotopy idempotent functor $L_f$ called $f$-localization,  \cite[Theorem 1.A.3]{Dror},  such that the coaugmentation
map $\eta: X \rightarrow L_f X$ is an $f$-local equivalence to an $f$-local space.

\begin{example}
\label{example:nullification}
When $f$ is a map of the form $A \rightarrow \ast$, one traditionally writes $\P_A$ for the localization functor $L_f$. This
functor is called $A$-nullification. It turns $A$ into a point and one says that $\P_A$ kills $A$.
For $A=S^{n+1}$ one obtains a functorial construction of the $n$-th Postnikov
section since $\P_{S^{n+1}} X \simeq X[n]$, \cite[Example~1.A.1.1]{Dror}. We will also be interested in $B \mathbb Z/p$-nullification.
\end{example}

\begin{example}
\label{example:classical2}
Let $f$ be a universal $H \mathbb F_p$-equivalence, meaning that $f$ is a wedge of all maps between countable simplicial
sets that induce an isomorphism in mod $p$ homology. Then $L_f$ is mod $p$ homological localization, \cite[Example~1.E.4]{Dror}.
The effect of $L_f$ on nilpotent spaces is, up to homotopy, Bousfield--Kan $p$-completion, \cite{MR0365573}.
\end{example}

We turn now to the description of the localization functor we need. We combine $B \mathbb Z/p$-nullification and $H \mathbb F_p$-localization so as to obtain a functor that ``kills" $B \mathbb Z/p$ and ``inverts" all mod $p$ homology equivalences.

\begin{definition}
\label{def:Neisendorfer}
Let $f: B \mathbb Z/p \rightarrow \ast$ and $g$ be a universal $H \mathbb F_p$-equivalence. We set $L = L_{f \vee g}$. \end{definition}

The reason behind this definition is that $L$ is a homotopy idempotent version of Neisendorfer's functor
$(\P_{B \mathbb Z/p} (-))^\wedge_p$, \cite[Section~1]{MR1321002},  $B \mathbb Z/p$-nullification followed by $p$-completion. Note that, by the universal properties of the functors, we have natural transformation of coaugmented functors $(\P_{B \mathbb Z/p} (-))^\wedge_p \rightarrow L $. This is induces a weak equivalence $\xymatrix{(\P_{B \mathbb Z/p} X)^\wedge_p  \ar[r]^-{\sim} & L X}$ whenever $\P_{B \mathbb Z/p} X$ is a nilpotent space since in that case we know that $(\P_{B \mathbb Z/p} X)^\wedge_p $ is $B \mathbb Z/p$-null.

\begin{remark}
\label{rem:Neisendorfer}
If $X$ is a connected infinite loop space with a torsion
fundamental group, then $L X$ is contractible. Indeed,
McGibbon proved that for such spaces,
$(\P_{B \mathbb Z/p} X)^\wedge_p$
is contractible, \cite[Theorem~2]{MR1371135}.  As  infinite loop spaces are nilpotent, we have also that $\P_{B \mathbb Z/p} X$  is nilpotent  and therefore,  by the comments above $ L X \simeq (\P_{B \mathbb Z/p} X)^\wedge_p  \simeq \ast  $.  In general however $L$ and
$(\P_{B \mathbb Z/p} (-))^\wedge_p$
differ.
\end{remark}

The functorial nature of homotopy localization is quite powerful, as is illustrated by the following central property.

\begin{theorem}\cite[Theorem~1.H.1]{Dror}
\label{prop:fiberwise}
If  $F \rightarrow E {\buildrel p \over \longrightarrow} B$ is a fibration and $L_f F \simeq \ast$, then $L_f (p) : L_f E \rightarrow L_f B$ is a homotopy equivalence.
\end{theorem}
We end this subsection by introducing cellularization functors. We fix a pointed space $A$. A pointed space $X$ is
$A$-\emph{cellular} if it belongs to the smallest class of pointed spaces containing $A$ and closed under weak equivalences
and pointed homotopy colimits, \cite[Definition~2.D.1]{Dror}. A pointed map $f$ is an $A$-\emph{equivalence} if
$\hbox{\rm map}_*(A, f)$ is a weak equivalence. There exists an augmented homotopy idempotent functor $\CW_A$ called $A$-cellularization
such that the augmentation map $\epsilon: \CW_A X \rightarrow X$ is an $A$-equivalence from an $A$-cellular space, \cite[Theorem~2.B.3]{Dror}.

\begin{example}
\label{example:covers}
When $A= S^{n+1}$, then $\CW_{S^{n+1}} X$ is a functorial analogue of $X \langle n \rangle$, the $n$-connected cover of $X$, \cite[Example~2.D.2.6]{Dror}.
Thus $\CW_{S^{n+1}} X$ coincides with the homotopy fiber of the coaugmentation
$X \rightarrow \P_{S^{n+1}} X$.
In particular, when $n=1$, we get
a functorial construction of the universal cover.
\end{example}

Farjoun shows in \cite[Theorems 3.A.1 and 3.A.2]{Dror} that localization and cellularization functor behave well with respect to loop space structures.
His statements and the proofs are more precise than the following theorem, which will be sufficient for us.

\begin{theorem}[Farjoun, \cite{Dror}]
\label{thm:localizeloop}
For any map $f$ and any pointed space $A$, the coaugmentation $\Omega X \rightarrow L_f \Omega X$ and
the augmentation $\CW_A \Omega X \rightarrow \Omega X$ are homotopic to loop maps, i.e. maps between loop spaces
preserving the loop space structure.
\end{theorem}


\subsection{The effect of $L$ on certain loop spaces}
\label{subsec:effect}
We work with the functor $L$ introduced in Definition~\ref{def:Neisendorfer}.
We first show that the effect of $L$ on loop spaces can be seen in the universal cover and second,
that the epfl of the universal cover is not greater than that of the space.

\begin{lemma}
\label{lem:cover}
Let $X$ be a connected space such that $\pi_1 X$ is a finite group, and let $X \langle 1 \rangle$ be the universal cover of $X$.
If $L( X \langle 1 \rangle)$ is contractible, then so is $LX$.
\end{lemma}

\begin{proof}
Consider the fibration sequence $X \langle 1 \rangle \rightarrow X\rightarrow B \pi_1 X$ given by the universal cover.
As we assume that $L (X \langle 1 \rangle)$ is contractible, Theorem~\ref{prop:fiberwise} implies that
$L X \simeq L(B\pi_1 X)$. But $\P_{B \mathbb Z/p} (B \pi_1 X)$ is contractible by \cite[Lemma 6.6]{MR1622342},
hence so are $L(B\pi_1 X)$ and $L X$.
\end{proof}

\begin{lemma}
\label{lem:fiber}
Let $f\colon Z \rightarrow Y$ be a map of connected spaces whose homotopy fiber $F$ is an infinite loop space. Then
the homotopy fiber of the induced map on universal covers $f\langle 1 \rangle\colon Z \langle 1 \rangle  \rightarrow Y \langle 1 \rangle$ is an infinite
loop space as well. Moreover, if $f$ is a principal fibration, then so is $f \langle 1 \rangle$.
\end{lemma}

\begin{proof}
Consider the horizontal ladder of fibration sequences
\[
\xymatrix{
\textrm{Fib} ( f \langle 1 \rangle)  \ar[r]\ar[d]_\psi & Z \langle 1 \rangle \ar[r]^-{f \langle 1 \rangle}\ar[d] & Y \langle 1 \rangle\ar[d] \\
F  \ar[r] & Z  \ar[r]^f & Y ,
}
\]
and the associated ladder of long exact homotopy sequences. By the Five Lemma, we obtain that
$\pi_r(\psi)$ is an isomorphism for $r\geq 2$, while the Snake Lemma shows that
$\pi_1(\psi)$ is a monomorphism. Moreover, since $\textrm{Fib} ( f \langle 1 \rangle)$ is connected,
the range of $\psi$ is in $F_c$, the base point component of $F$, and the long exact homotopy sequence of the fibration
\[
\xymatrix{
\textrm{Fib} \ \psi  \ar[r] & \textrm{Fib} ( f \langle 1 \rangle)\ar[r]^(.65)\psi & F_c
}
\]
shows that $\textrm{Fib} \ \psi $ is a homotopically discrete space. Then,  $\textrm{Fib} ( f \langle 1 \rangle)$ is a covering of $F_c$, and therefore it is a loop space since infinite loop structures are
preserved when considering connected components and covers (see Theorem~\ref{thm:localizeloop}). Moreover, if $f$ is a principal fibration, so is $f \langle 1 \rangle$,  since taking universal covers is a particular case of cellularization (see
Example~\ref{example:covers}) and,  the cellularization of a principal fibration is again a principal fibration  by  \cite[Theorem~2.1]{MR2581908}.
\end{proof}

\begin{corollary}
\label{cor:epflcover}
Let $Z$ be a pointed connected space. Then, $\epfl {Z \langle 1 \rangle} \leq \epfl Z.$

\end{corollary}

\begin{proof}
Suppose that $Z$ is an extension by principal fibrations of length $n$ and apply, to the associated tower of principal fibrations,
the universal cover functor. Then, the result follows by Lemma \ref{lem:fiber}.
\end{proof}

We refine this elementary observation to obtain a version where the homotopy fibers in the tower are simply connected,
a key technical fact for what follows.
\begin{lemma}
\label{lem:fibercover}
Let $f\colon Z \rightarrow Y$ be a principal fibration of simply connected spaces whose homotopy fiber $F$ is an infinite loop space.
There exists then a factorization $f\colon Z \rightarrow \overline{Y} \rightarrow Y$ such that the homotopy fiber 
$\textrm{Fib}(Z \rightarrow \overline{Y})$ is $F\langle 1 \rangle$, the universal cover of $F$.
\end{lemma}

\begin{proof}
Let $\theta\colon Y\to BF$ denote the classifying map of the principal fibration $f$. Since $F$ is connected,
$BF$ is simply connected and the second stage of the Postnikov tower becomes
\[
\xymatrix{
B(F \langle 1 \rangle) \simeq (BF) \langle 2 \rangle  \ar[r] & BF \ar[r]^-{\kappa_2} & K(\pi_2 BF, 2).
}
\]
Let $\overline{Y}$ be the homotopy fiber of ${\kappa_2}\circ \theta\colon Y\to  K(\pi_2 BF, 2)$. Since $ {\kappa_2} \circ \theta \circ f \simeq\ast$,
then $f$ factors through $\overline{f}\colon Z\to \overline{Y}$. Loop now the existing fibration sequences to obtain the following commutative diagram of
horizontal and vertical fibration sequences:
\[
\xymatrix{
  & \Omega Y \ar@{=}[r]\ar[d] & \Omega Y \ar[d]\\
F \langle 1 \rangle\ar[r]\ar[d] & F \ar[r]\ar[d] & K(\pi_2 BF, 1)\ar[d]\\
\textrm{Fib} \ \overline{f}  \ar[r] & Z \ar[r]^{\overline{f}} & \overline{Y}
}
\]
Since the lower right square is a homotopy pull-back (principal fibrations with the same vertical fiber), the map between the 
horizontal fibers $F \langle 1 \rangle \to \textrm{Fib} \ \overline{f} $ is a homotopy equivalence.
\end{proof}

The following result is the final step that will allow us to understand the effect of the functor $L$ on a space $Z$ with finite epfl.

\begin{proposition}
\label{prop:cover}
Let $Z$ be a connected space with $\epfl {Z} \leq n $.
Then,  $Z \langle 1 \rangle$ is also an extension by (not necessarily principal) fibrations of length $n$, where the fibers
are  simply connected infinite loop spaces.
\end{proposition}

\begin{proof}
Since $\epfl {Z} \leq n$ there exists a tower of principal fibrations
\[
\xymatrix{
Z \simeq Z_n  \ar[r] & Z_{n-1} \ar[r] & \dots \ar[r] & Z_1 \ar[r] & \ast
}
\]
such that the homotopy fibers $\textrm{Fib}(Z_{k} \rightarrow Z_{k-1})$ are infinite loop spaces. Since $Z$ is
connected the spaces $Z_k$ can be chosen to be connected as well (notice that fibrations are, in particular, 
surjective maps). Lemma~\ref{lem:fiber} implies that the
tower of universal covers is made of principal fibrations that have also infinite loop spaces as homotopy fibers. 
We are thus left with a tower of simply connected spaces
\[
\xymatrix{
Z\langle 1 \rangle  \simeq Z_n \langle 1 \rangle  \ar[r] & Z_{n-1} \langle 1 \rangle  \ar[r] & \dots \ar[r] & Z_1 \langle 1 \rangle \ar[r] & \ast
}
\]
The fibers $F_k$ however are not simply connected in general, but only connected.
We modify thus the spaces $Z_k \langle 1 \rangle$ for $k < n$ to find a different and more convenient tower for $Z\langle 1 \rangle$. 
Since $F_1 =  Z_1 \langle 1 \rangle$ is simply connected, we can assume by induction
that $Z_{n-1} \langle 1 \rangle$ is an extension by fibrations  of length $n-1$, with simply connected infinite loop spaces as fibers.

We use now the factorization $Z_n \langle 1 \rangle \rightarrow \overline{{Z}_{n-1} \langle 1 \rangle }\rightarrow Z_{n-1} \langle 1 \rangle$ of Lemma~\ref{lem:fibercover}.
Since the homotopy fiber of the first map is a simply connected infinite loop space, to conclude the proof
we must show that  the space in the middle is an extension by fibrations of length  $n-1$
with simply connected infinite loop spaces as fibers. Consider now the homotopy fiber $H$ of the composite
map $\overline{{Z}_{n-1} \langle 1 \rangle }  \rightarrow Z_{n-1} \langle 1 \rangle \rightarrow Z_{n-2} \langle 1 \rangle$.
The homotopy fiber $H$ fits by construction into a fibration sequence
\[
H \rightarrow F_{n-1} \rightarrow K(\pi_2 BF_n, 2).
\]
However, since $F_{n-1}$ is simply connected by induction hypothesis, the map $F_{n-1} \rightarrow K(\pi_2 BF_n, 2)$
factors through the second Postnikov section of $F_{n-1}$. This is a map of infinite loop spaces, see Theorem~\ref{thm:localizeloop},
which implies that $H$ is an infinite loop space. Therefore, the tower $ \overline{{Z}_{n-1} \langle 1 \rangle }  \rightarrow   Z_{n-2} \langle 1 \rangle \rightarrow \dots \rightarrow  Z_{1} \langle 1 \rangle$
exhibits $ \overline{{Z}_{n-1} \langle 1 \rangle }   $ as an extension by fibrations of length $n-1$ with \emph{connected} infinite loop spaces as fibers.
The induction hypothesis allows us to conclude that it is also an extension by principal fibrations of length $n-1$ with \emph{simply connected}
infinite loop spaces as fibers.\end{proof}

We finally describe the effect of $L$ (see Definition \ref{def:Neisendorfer}) on loop spaces of finite extension by principal fibrations:

\begin{theorem}
\label{thm:Lhomotopynilpotent}
Let $\Omega Z$ be a connected space with finite fundamental group. If $\epfl Z$ is finite, then  $L( \Omega Z)$ is contractible.
\end{theorem}
\begin{proof}
By Lemma~\ref{lem:cover} it is enough to prove
that $L\big( (\Omega Z)\langle 1 \rangle \big)$ is contractible.  Let $\epfl Z \leq n$. Then, since looping a tower of fibrations with
infinite loop spaces fibers yields another such tower, $\epfl {\Omega Z}\leq n$ and
by Proposition~\ref{prop:cover}, $(\Omega Z)\langle 1 \rangle $ is also an extension of fibrations of length $n$,
\begin{equation*}
(\Omega Z)\langle 1 \rangle \simeq Y_n \rightarrow Y_{n-1}  \rightarrow \dots \rightarrow Y_{1} \rightarrow \ast
\end{equation*}
where the fibers are simply connected infinite loop spaces~$F_k$.
By  Remark  \ref{rem:Neisendorfer},
$L F_k$ is contractible for any $n \geq k \geq 1$, and according to Theorem~\ref{prop:fiberwise} $L(Y_k)\simeq L(Y_{k-1})$.
Therefore, an inductive argument shows that $L\left( (\Omega Z)\langle 1 \rangle \right) \simeq \ast$.
\end{proof}

\begin{remark}
\label{rem:tori}
The finiteness assumption on the fundamental group in Theorem~\ref{thm:Lhomotopynilpotent} is not necessary, but
the contractibility of  $L( \Omega Z)$ does not hold as soon as there is a copy of the integers in $\pi_1 (\Omega Z)$.
Suppose for example that $\pi_1 (\Omega Z)$ is a finitely generated abelian group which is infinite. There exists then an epimorphism to $\mathbb Z$
which can be used to construct a map $\Omega Z \rightarrow K\big(\pi_1 (\Omega Z), 1\big) \rightarrow K(\mathbb Z, 1) = S^1$. This map has a section hence $\Omega Z$ dominates $S^1$. Therefore $L(\Omega Z)$ dominates $LS^1 \simeq (S^1)^\wedge_p$, and $L(\Omega Z)$ is not contractible.
\end{remark}

\section{Homotopy nilpotency of $p$-complete loop spaces}
\label{sec:modp}

In this section we characterize homotopy commutative loop spaces in the sense of Biedermann--Dwyer, and more generally homotopy nilpotent groups with finiteness conditions.

We now recall some definitions. A loop space $(X, BX)$ is said to be a $p$-compact group  \cite[Definition 2.3]{DW1} if $BX$ is $p$-complete and $X$ is $\F_p$-finite, i.e., \ $H^*(X;{\F}_p)$ is a finite dimensional ${\F}_p$-vector space. A $p$-compact torus of rank $r$ is a loop space $(T, BT)$ such that $BT$ is an Eilenberg-MacLane space of type $K\big((\Z^\wedge_p)^r, 2\big)$ \cite[Definition 6.3]{DW1}.
The following result characterizes homotopy nilpotent $p$-compact groups in the Biedermann--Dwyer
sense.

\begin{theorem}
\label{thm:C}
Let $(X, BX)$ be a connected $p$-compact group. Then, $\nil{X}$ is finite if and only if $(X,BX)$ is a $p$-compact torus.
\end{theorem}

\begin{proof}
By applying the cellularization functor $\CW_{S^2}$ to $X$,
we obtain the universal cover of $X$,
$\CW_{S^2} (X) = X \langle 1 \rangle  \simeq \textrm{Fib}(X \rightarrow B\pi_1(X))$.
Cellularization is a continuous functor that preserves products up to homotopy \cite[Theorem~2.E.10]{Dror}, therefore if $\nil{X}$ is finite, then
so is $\nil{X \langle 1 \rangle}$.
Now, since $X \langle 1 \rangle $ is a connected loop space
with finite fundamental group,  Theorem~\ref{thm:Lhomotopynilpotent} tells us that $L (X \langle 1 \rangle )$ must be contractible.

However, by Miller's proof of the Sullivan conjecture \cite[Theorem~A]{Miller} the pointed mapping space
$\hbox{\rm map}_*(B\mathbb Z/p, X\langle 1 \rangle)$ is contractible. In terms of localization functors, this means that
$X\langle 1 \rangle$ is $B\mathbb Z/p$-local, and so $\P_{B \mathbb Z/p}\big( X\langle 1 \rangle \big) \simeq X\langle 1 \rangle$,
as noticed already in the introduction of \cite{MR1321002}. Therefore, since $X \langle 1 \rangle$ is $p$-complete,
it is in fact $(f \vee g)$-local,
which means that $L \big( X \langle 1 \rangle\big) \simeq X \langle 1 \rangle$. Thus,
$X \langle 1 \rangle$ must be contractible, so that $X$ is homotopy equivalent to $B\pi_1(X)$.

Now, by
\cite[Remark~2.2]{DW1}, the fundamental group of a $p$-compact group is
a finite direct sum of copies of cyclic groups $\Z/p^r$ and copies of $p$-adic integers $\mathbb Z^\wedge_p$.
Since $X$ is $\mathbb F_p$-finite, and $B\Z/p^r$ is not, this implies that $\pi_1(X)$ contains no factor of $\Z/p^r$-type.
We conclude that $(X,BX)$ is a $p$-compact torus.
\end{proof}

For any simply connected Lie group $G$, the $p$-completion of $BG$ gives rise to a $p$-compact group.
The following result is thus straightforward.

\begin{corollary}
\label{cor:infinite}
Let $G$ be a non-trivial simply connected compact Lie group. Then $\nil{G^\wedge_p}$ is infinite.
\hfill{$\square$}
\end{corollary}

This implies that the nilpotency in the sense of Berstein--Ganea, and the nilpotency in the sense of
Biedermann--Dwyer do not coincide in general.

\begin{example}
\label{ex:kishimoto}
McGibbon's $p$-local examples of classically homotopy commutative compact Lie groups can be translated into the $p$-complete
setting, \cite[Theorem 2]{McGibbon}. Simply connected simple Lie groups at large enough primes, $Sp(2)$ at the prime~$3$,
and the exceptional group $G_2$ at the prime~$5$, have Berstein--Ganea
nilpotency~$1$, whereas according to Corollary~\ref{cor:infinite} they all have infinite Biedermann--Dwyer nilpotency.

In \cite{KK} Kaji and Kishimoto  study examples of $p$-compact groups that are nilpotent in the sense of Berstein--Ganea
but they are not  $p$-compact tori, so by Theorem~\ref{thm:C} they are not homotopy nilpotent in the sense of Biedermann--Dwyer.
There exist in fact infinitely many loop spaces $X$  for which $\nilbg X  < \infty$, but $\nil X = \infty$. This is the case
for $E_8$ at the prime $41$ with $\nilbg{(E_8)^\wedge_{41}} = 3$, \cite[Theorem~1.6]{KK}.
\end{example}

We move now from $p$-compact groups to $p$-Noetherian groups. Let us recall from \cite{CCS_2} that
a $p$-Noetherian group is a loop space $(X, BX)$ where $BX$ is $p$-complete and $H^\ast (X; \F_p)$
is a finitely generated (Noetherian) $\F_p$-algebra.

\begin{theorem}
\label{thm:B}
Let $(X, BX)$ be a connected $p$-Noetherian group.  If $X$ is a homotopy nilpotent group, then $BX$ fits in a fibration sequence
\[
K ( Q, 2) \times K ( \Z^\wedge_p, 3)^r \rightarrow BX \rightarrow  \left((BS^1)^\wedge_p\right)^s
\]
where $Q$ is a finite abelian $p$-group, and $r,s\geq 0$.
\end{theorem}

\begin{proof}
In \cite[Theorem 1.9]{CCS_2}, for a $p$-Noetherian group $X$, the authors construct a fibration
\[
K(P, 2 )^\wedge_p\rightarrow BX \rightarrow BY = (\P_{\Sigma B \Z/p} BX)^\wedge_p
\]
where $P=Q\oplus (\Z/p^\infty)^r$, for $Q$ a finite abelian $p$-group, $r\geq 0$, $\Z/p^\infty = \mathbb Z[1/p]/ \mathbb Z$
is a Pr\"ufer group, \cite[Theorem~10.13]{MR1307623}, and $(Y,BY)$ is a $p$-compact group.
Since we have a series of homotopy equivalences
\[
Y \simeq \Omega B Y  \simeq \Omega (\P_{\Sigma B \Z/p} BX)^\wedge_p \simeq \left(\Omega (\P_{\Sigma B \Z/p} BX)\right) ^\wedge_p
\]
and $\left(\Omega (\P_{\Sigma B \Z/p} BX)\right) ^\wedge_p$ is homotopy equivalent to  $(\P_{B\Z/p} \Omega B X ) ^\wedge_p$ by \cite[Theorem~3.A.1]{Dror},
we get that $Y \simeq (\P_{B\Z/p} X)^\wedge_p$. Also observe that all the homotopy equivalences are loop maps by Theorem~\ref{thm:localizeloop}.
This allows us to say that  $(Y, BY)$ is obtained in a functorial way from $(X, BX)$ using the nullification and the $p$-completion functors.
Both functors are continuous and preserve products up to homotopy.
Therefore, if $\nil{X}$ is finite, so is $\nil{Y}$.

Finally, since $(Y,BY)$ is a $p$-compact group and $\nil{Y}$ is finite, Theorem~\ref{thm:C} implies that $(Y,BY)$ must be a $p$-compact torus,
that is, $BY\simeq \left((BS^1)^\wedge_p\right)^s$ for some $s\geq 0$.
For $\Z/p^\infty$,  $p$-completion shifts dimension by one, \cite[Example~VI.6.4]{MR0365573}:
\[
K(P, 2 )^\wedge_p\simeq K(Q,2)\times K\big((\Z/{p^\infty})^r,2\big)^\wedge_p \simeq K ( Q, 2) \times K ( \Z^\wedge_p, 3)^r
\]
and the result follows.
\end{proof}

It follows from the previous theorem that $p$-Noetherian groups which are homotopy nilpotent have epfl at most $2$, where the tower of principal fibrations with infinite loop space fibers is provided by the Postnikov tower.
We prove now that this invariant coincides with the Biedermann--Dwyer nilpotency in this situation.

\begin{corollary}
\label{cor:thmB}
Let $(X, BX)$ be a connected $p$-Noetherian group and assume that $\nil{X}$ is finite. Then $\nil{X} \leq2$.
\end{corollary}

\begin{proof}
By the previous theorem, the homotopy groups $\pi_n(BX)$ vanish unless $2 \leq n \leq 3$.
The non-trivial homotopy groups of the simply connected space $BX$ live thus in the metastable range, therefore
by \cite[Example 9.8]{MR2580428}, $\nil{X} \leq 2$.
\end{proof}

We finally come back to the \emph{integral} Torus Theorem, \cite[Theorem~1.1]{Hubbuck}. The classical statement is that, if $\nilbg X = 1$,
then $X$ has the homotopy type of a torus.

\begin{theorem}
\label{thm:A}
Let $(X, BX)$ be a connected finite loop space. If $\nil{X}$ is finite, then  $X$ has the homotopy type of a torus.
\end{theorem}

\begin{proof}
For a finite and connected loop space $(X, BX)$, it is well known that the $p$-completion $(X^\wedge_p, B X^\wedge_p )$ is a
$p$-compact group for every prime $p$ (see, for example, \cite[Introduction]{MR2135183}, \cite[p. 990]{MR2827828}).
Theorem \ref{thm:C} applies and we obtain that $X^\wedge_p$ is a $p$-compact
torus for every prime $p$. Now, Sullivan's arithmetic square, \cite[Theorem VI.8.1]{MR0365573},
\[
\xymatrix{
X \ar@{->}[r]\ar[d] & \prod_p X^\wedge_p \ar[d]\\
X_0 \ar[r] & \left( \prod_p X^\wedge_p \right)_0
}
\]
is a  homotopy pull-back square. On the right hand side we have a product of $p$-complete tori and its rationalization, whereas
on the bottom left corner the rationalization of a finite loop space is a product of odd dimensional rational spheres $K(\mathbb Q, 2k+1)$.
Any higher dimensional sphere than $S^1$ would remain in the homotopy pull-back, so $X_0$ must be a rational torus
and we conclude that $X$ has the homotopy type of a torus.
\end{proof}

Following the ideas of Rector \cite{Rector}, it is commonly accepted that compact Lie groups should be thought of
as finite loop spaces, and that the structural data of the Lie group have to be read by means of homotopy invariants.
Within this framework, connected non abelian compact Lie groups are expected to be highly non nilpotent.
The previous result shows that nilpotency in the sense of Biedermann--Dwyer is the right notion in contrast with
classical nilpotency in the sense of Berstein--Ganea. So for example, by a classical result of Porter \cite{MR0169244}
the sphere $S^3$ is nilpotent in the sense of Berstein--Ganea,
$\nilbg {S^3} =3$, but since $S^3$ is not a torus, $\nil {S^3} = \infty.$



\bibliographystyle{amsplain}
\providecommand{\bysame}{\leavevmode\hbox to3em{\hrulefill}\thinspace}
\providecommand{\MR}{\relax\ifhmode\unskip\space\fi MR }
\providecommand{\MRhref}[2]{%
  \href{http://www.ams.org/mathscinet-getitem?mr=#1}{#2}
}
\providecommand{\href}[2]{#2}



\medskip
\begin{minipage}[t]{8 cm}
Cristina Costoya\\
UDC Computaci\'on\\
Edificio \'Area Cient\'ifica, 3.04\\
15071 A Coru\~na, Spain\\
\textit{E-mail:}\texttt{\,cristina.costoya@udc.es}\\
\end{minipage}
\begin{minipage}[t]{8 cm}
J\'er\^ome Scherer\\
EPFL SB MATHGEOM\\
Station 8, MA B3 455\\
CH - 1015 Lausanne, Switzerland\\
\textit{E-mail:}\texttt{\,jerome.scherer@epfl.ch}\\
\end{minipage}

\medskip
\begin{minipage}[t]{8cm}
Antonio Viruel\\
UMA \'Algebra, Geometr\'ia y Topolog\'ia\\
Campus Teatinos \\
29071 M\'alaga, Spain\\
\textit{E-mail:}\texttt{\,viruel@uma.es}\\
\end{minipage}

\end{document}